\documentclass[10pt]{article}
\usepackage[utf8]{inputenc}

\usepackage[left=1.4in,top=1.3in,right=1.3in,bottom=1.3in,footskip = 0.333in]{geometry}

\usepackage{authblk, amsmath, amsthm, amsfonts, amssymb, amsbsy, bbm, bigstrut, graphicx, enumerate,  upref, longtable, comment, comment,siunitx,mathrsfs,bm}
\usepackage{float}
\usepackage{tabularx, booktabs, makecell, caption}
\usepackage{siunitx}
\usepackage{algorithm}
\usepackage[dvipsnames]{xcolor}

\usepackage{subcaption}
\usepackage[breaklinks]{hyperref}

\usepackage[numbers, sort&compress]{natbib} 
\usepackage{comment}
\usepackage[export]{adjustbox}

\newtheorem{theorem}{Theorem}
\newtheorem{lemma}{Lemma}

\newtheorem{prop}{Proposition}

\theoremstyle{definition}

\newtheorem{ass}{Assumption}

\numberwithin{theorem}{section}
\numberwithin{remark}{section}
\numberwithin{ex}{section}
\numberwithin{ass}{section}
\numberwithin{defn}{section}
\numberwithin{lemma}{section}
\numberwithin{corollary}{section}
\numberwithin{prop}{section}
\numberwithin{conjecture}{section}

\newcommand{\argmin}{\operatorname{argmin}}

\newcommand{\bG}{\mathbf{G}}

\newcommand{\bc}{\mathbf{c}}
\newcommand{\bA}{\mathbf{A}}
\newcommand{\by}{\mathbf{y}}

{\tiny {\tiny }}

\newcommand{\bct}{\widehat{\bc^\top \btheta}}

\newcommand{\rr}{\mathbb{R}}

\allowdisplaybreaks

\def \bb{\mathbf{b}}
\def \be{\mathbf{e}}
\def \bu{\mathbf{u}}
\def \bM{\mathbf{M}}
\def \bz{\mathbf{z}}
\def \bX{\mathbf{X}}
\def \be{\mathbf{e}}

\def \bu{\mathbf{u}}

\def \btheta {\boldsymbol{\theta}}
\def \bSigma {\boldsymbol{\Sigma}}

\makeatletter
\def\thickhline{%
	\noalign{\ifnum0=`}\fi\hrule \@height \thickarrayrulewidth \futurelet
	\reserved@a\@xthickhline}
\def\@xthickhline{\ifx\reserved@a\thickhline
	\vskip\doublerulesep
	\vskip-\thickarrayrulewidth
	\fi
	\ifnum0=`{\fi}}
\makeatother

\newlength{\thickarrayrulewidth}
\usepackage{algorithm}
\numberwithin{equation}{section}

\DeclareMathOperator{\sech}{sech}

\begin{document}
\title{Inference in high-dimensional logistic regression under tensor network dependence}

\date{}

\author[1]{Josh Miles\thanks{miles.j@ufl.edu}}
\author[1]{Sohom Bhattacharya\thanks{bhattacharya.s@ufl.edu}}
\affil[1]{Department of Statistics, University of Florida}
\makeatletter
\renewcommand\AB@affilsepx{\\ \protect\Affilfont}
\renewcommand\Authsep{, }
\renewcommand\Authand{, }
\renewcommand\Authands{, }
\makeatother

\maketitle

\begin{abstract}
   We investigate the problem of statistical inference for logistic regression with high-dimensional covariates in settings where dependence among individuals is induced by an underlying Markov random field. Going beyond  the pairwise interaction models such as the Ising model, we consider a framework to accommodate more general tensor structures that capture higher-order dependencies. We develop a two-step procedure for low-dimensional linear and quadratic functionals. The first step constructs a regularized maximum pseudolikelihood estimator, for which we establish consistency under high-dimensional features. However, as in other classical high-dimensional regression problems, this estimator is biased and cannot be directly used for valid statistical inference. The second step introduces a bias-correction that yields an asymptotically normal estimator from which one can construct confidence intervals and test hypotheses. Our results move beyond the existing literature, where only estimation guarantees were available or only for pairwise interaction models. We complement our theoretical analysis with simulation studies confirming the effectiveness of the proposed method.
\end{abstract}


\section{Introduction}

Regression with dependent observations has gained considerable attention in recent network analysis research~\citep{mukherjee2024logistic,daskalakis_2020}. Given a network, the primary objective in this framework is to efficiently estimate the regression coefficients while accounting for dependencies among observations. 
The problem is particularly relevant in scenarios where the standard independence assumption on the response variables is violated. Such dependencies commonly arise in settings where data is collected in temporal or spatial domains or within social networks, where interactions between individuals introduce correlation structures through peer effects~\citep{bertrand2000network,duflo2003role,sacerdote2001peer,trogdon2008peer}. While recent work has advanced understanding in such models, most results are restricted to low-dimensional covariates or focus only on consistent estimation. As a result, inference for high-dimensional regression under dependence remains poorly understood. This paper addresses this important gap by studying statistical inference for low-dimensional functionals in high-dimensional logistic regression with dependent binary outcomes, where dependence is induced by an underlying hypergraph.

To achieve this, we consider a discrete exponential family of binary outcomes, where individuals are connected via 
a Markov Random Field. 
A widely studied model considers quadratic interactions, which correspond to classical Ising model~\cite{ising1925beitrag}. Originally introduced in statistical physics to study ferromagnetism, the Ising model has since found broad applications in spatial statistics, social networks, computer vision, and neuroscience~\cite{banerjee2003hierarchical,geman1986markov,green2002hidden,montanari2010spread}. Further, the Ising model with external fields has been a popular framework for modeling logistic regression with peer effects~\cite{daskalakis2019regression,mukherjee2024logistic}. Related problems of consistent estimation and hypothesis testing from a single realization of an Ising model have been studied extensively in the past couple of decades~\cite{chatterjee2007estimation,bhattacharya2018, mukherjee2018global,mukherjee2022testing}. However, in many real-world datasets, pairwise interactions are insufficient to capture complex dependency structures; higher-order peer-group effects must also be considered. Motivated by this, we study hypothesis testing under general hypergraph-induced dependence.

Inference in high-dimensional linear and generalized linear models (GLMs) has been well- studied in the last decade. It is known~\cite{buhlmann2011statistics,van2014asymptotically} that while regularized estimators such as the Lasso achieve consistent estimation, they cannot be directly used for hypothesis testing. A substantial line of work has therefore focused on debiasing such estimators to obtain asymptotically normal statistics for inference in both linear regression~\cite{zhang2014confidence,javanmard2014hypothesis,javanmard2018debiasing} and GLMs~\cite{cai_debiased_clt,ma_2020,ma2024statistical}. Importantly, all of these results assume independent samples. In contrast, our work develops valid confidence intervals for low-dimensional functionals in high-dimensional logistic regression with tensor network dependence, thus broadening the applicability for a wide class of dependent data models.

Our contributions can be summarized as follows. We study the problem of statistical inference and hypothesis testing in high-dimensional logistic regression models with dependent observations induced by an underlying tensor network structure. We propose a bias-corrected estimator (Section~\ref{sec:lin_func}) based on the $\ell_1$-regularized maximum pseudo-likelihood estimator~\eqref{eq:define_theta_isl}, which is computationally efficient even under higher-order dependence. The proposed estimator is asymptotically normal and enables a valid construction of confidence intervals for both linear and quadratic functionals of the regression parameter. We establish non-asymptotic error bounds for the pseudo-likelihood estimator and provide rigorous justification of the inference procedure through Theorem~\ref{thm:main} and Proposition~\ref{prop:CI}. Our proof-techniques can be of independent interest for high-dimensional nonlinear statistical inference problems for network data. To the best of our knowledge, this is the first work to demonstrate the validity of projection-based inference~\eqref{eq:def-uhat} under general network and tensor dependence.

\subsection{Related works}
Our work lies at the intersection of two key research areas: regression analysis on network data and inference of low-dimensional functionals in high-dimensional statistical models.

Estimating model parameters from a single sample of binary outcomes observed on an underlying network has received growing attention.
For the classical Ising model~\cite{ising1925beitrag}, this problem has been extensively investigated, beginning with the foundational works of~\cite{comets1992consistency,gidas1988consistency,guyon1992asymptotic}, which established consistency and optimality of maximum likelihood estimators when the underlying network follows a spatial lattice structure.
Because the maximum likelihood estimator is computationally intractable for general interaction matrices,~\cite{besag1974spatial,besag1975statistical} introduced the maximum pseudo-likelihood estimator (MPLE), whose statistical properties have been rigorously analyzed for binary Markov random fields~\citep{chatterjee2007estimation,bhattacharya2018,ghosal2020joint}.
More recently,~\citep{daskalakis2019regression,daskalakis_2020,mukherjee2024logistic} considered logistic regression models on networks, where the dependence among observations arises through the network structure.
Extensions to tensor-type dependencies in binary data have been studied in~\citep{mukherjee2022estimation,liu2024tensor}.
Several related works address hypothesis testing and inference problems for the Ising model without covariates, or with low-dimensional covariates~\citep{berthet2019exact,bhattacharya2025nonsense,lowe2020exact,neykov2019property}.
To our knowledge, our paper is the first to address statistical inference in high-dimensional logistic regression under general higher-order (hypergraph) dependencies.

The question of consistent estimation in high-dimensional regression with independent observations has been extensively studied~\cite{bach2010self,huang2012estimation,negahban2012unified,plan2012robust,van2008high,meier2008group}.
These works primarily focus on the traditional high-dimensional regime with independent observations, where the number of covariates can exceed the sample size, and the true regression parameter is assumed to be sparse.
A major line of research has focused on bias-correction methods for regularized estimators, which enable valid statistical inference and hypothesis testing in both linear and logistic models~\citep{zhang2014confidence,belloni2016post,shi2021statistical,zhu2020high,ma_2020}.
Beyond coordinate-wise inference (e.g., testing a specific $\theta_j$), more recent studies have developed inference procedures for low-dimensional functionals of high-dimensional parameters~\citep{guo2021inference,guo2019optimal,ma2024statistical,guo2022moderate}.
However, all of these works rely on the assumption of independent observations.
Motivated by applications with peer-effects, our work departs from this framework by developing inference procedures that remain valid under general dependent data structures induced by a tensor network. Our paper is technically most closely related to~\cite{mukherjee2024logistic}, which studies logistic regression under network dependence; however, their analysis is limited to pairwise interactions and does not address the statistical inference problem.

\noindent \textbf{Organization:} The rest of the paper is structured as follows: we present our results in Section~\ref{sec:results}. In particular, Section~\ref{sec:lin_func} provides our main method for statistical inference for linear functional. Section~\ref{sec:testing} and Section~\ref{sec:quadratic} delineates how our construction can be used for hypothesis testing and statistical inference of quadratic functional. Theoretical guarantees of our proposed method is provided in Section~\ref{sec:theory}. Numerical experiments and discussions are provided in Section~\ref{sec:sims} and Section~\ref{sec:discussion} respectively. All proofs of theoretical results are deferred to the Appendix.

\noindent \textbf{Notation:} We denote by $[n]$ the set $\{1,\ldots, n\}$. Given a subset $S \subseteq [n]$ and a vector $\bz \in \mathbb R^n$, define $\bz_S= (z_i)_{i \in S}$ and $\bz_{-S}= (z_i)_{i \notin S}$. For $j \in [d]$, $\be_j \in \mathbf{R}^d$, denote the canonical basis of $\mathbb{R}^d$. For any vector $\bu \in \mathbb{R}^d$, $\|\bu\|_0,\|\bu\|_1,\|\bu\|_2,$ and $\|\bu\|_\infty$ denote its $\ell_0, \ell_1, \ell_2, \ell_\infty$ norms, respectively. Given a graph $\bG_n$ on $n$ vertices, $\bA(\bG_n)$ denotes its adjacency matrix. Given a $d\times d$ matrix $\mathbf{B}_n$, denote by $\lambda_{\max}(\mathbf{B}_n)$ and $\lambda_{\min}(\mathbf{B}_n)$ its largest and smallest eigenvalues, respectively. We denote by $\text{Tr}(\mathbf{B}_n)$ the trace of $\mathbf{B}_n$. $z_{\alpha}$ and $\Phi$ denote the $\alpha$-quantile and CDF of a standard normal random variable. For two sequences of real numbers $a_n, b_n>0, a_n=O(b_n)$ will denote that $\limsup_{n \rightarrow \infty} \ a_n / b_n = C$ for some $C \in [0,\infty)$, $a_n=\Omega(b_n)$ will denote that $\liminf_{n \rightarrow \infty} \ a_n / b_n = C$ for some $C \in (0,\infty)$, and $a_n=o(b_n)$ will denote that $\lim_{n \rightarrow \infty} a_n/b_n=0$. We write $a_n \gg b_n$ if $a_n/b_n \rightarrow \infty$ and $a_n \lesssim b_n$ if there exists $C>0$ such that $a_n \le C b_n$. For any two sequences of random variables $Z_n$ and $W_n$ with $W_n>0$, we write $Z_n= O_P(W_n)$ and $Z_n= o_P(W_n)$ if $Z_n/W_n$ is tight and $Z_n/W_n \xrightarrow{\mathbb{P}} 0$, respectively.


\section{Results}\label{sec:results}

Consider the canonical Markov Random Field (MRF) model induced by a hypergraph $\bG_n=([n],E)$. The conditional p.m.f.  
of the response $\by$ given the covariates is given by
\begin{align}\label{model}
  \mathbb{P}_{\theta}(\by|\bX) = \frac{1}{Z(\btheta)} \exp \left(\sum_{\be \in E} g_\be \by_\be + \sum_{i=1}^{n} y_i \bX^\top_i \btheta\right), \qquad \by \in \{ \pm 1\}^n,
\end{align}
for some $g_{\be} \ge 0$, $\by_\be= \prod_{i \in \be} y_i$, and $Z(\btheta)$ denote the normalizing constant
\begin{align}\label{eq:define_Z}
    Z(\btheta)= \sum_{\bz \in \{ \pm 1\}^n} \exp \left(\sum_{\be \in E} g_\be \bz_\be + \sum_{i=1}^{n} y_i \bX^\top_i \btheta\right)
\end{align}
Here 
the observed covariates are denoted by the matrix
\(\boldsymbol X^\top = (\bX_1, \ldots, \bX_n)  \in \mathbb R^{d \times n}\). We will assume throughout that the covariates $\bX_i$ are i.i.d. with common distribution $\mathbb{P}_X$.
Our parameter of interest is $\btheta \in \mathbb{R}^d$. We will consider the traditional high-dimensional scenario where $d \rightarrow \infty$ as $n \rightarrow \infty$ but the parameter is sparse. Given a single realization $(\by,\bX)$ of the model, the main objective of the paper is to perform inference on the unknown regression parameter $\btheta$.

\subsection{Inference for linear functionals}\label{sec:lin_func}
In this section, we describe our method of performing inference on a linear functional $\bc^\top \btheta$, for some $\bc \in \rr^d$. We will first obtain a consistent estimator of $\btheta$, and then perform a bias-correction to obtain an estimator $\widehat{\bc^\top \btheta}$ which is asymptotically normal. We describe the details of our two-step method below.

Estimation of parameters in MRFs is challenging due to the fact that the normalizing constant $Z(\btheta)$, defined by~\eqref{eq:define_Z} is computationally intractable. As a viable alternative, pseudo-likelihood based estimators~\cite{besag1975statistical, chatterjee2007estimation} have been quite popular for Ising models. However, the majority of prior works that deal with consistent estimation are restricted to pair-wise interaction. Moreover, inferential properties of pseudo-likelihood-based estimators are poorly understood.
To construct a consistent estimator $\tilde\btheta$ of $\btheta$, we proceed as follows.
Given a hypergraph $\bG_n$, we call a subset of vertices $I \subset [n]$ a strong independent set~\citep{cohen2022number} if, for every hyperedge $\be \in E$, we have $|I \cap \be| \le 1$.
Note that for a graph $\bG_n$, this definition coincides with that of a standard independent set.
We begin by choosing a strong-independent set $S:= S(\mathbf{G}_n)$ of the hypergraph $\mathbf{G}_n$. We will split vertices in $S$ into two parts $S= S_1 \cup S_2$ of equal size (for simplicity, we will assume throughout that $|S|$ is even, otherwise one of the parts will have one extra vertex). We will use the vertices $S_i$, $i \in \{1,2\}$ during the $i$-th step of our method. Note that, conditional on $\by_{-S},\bX$, the random variables $y_i$, $i \in S$ are independent by property of MRF. Our initial estimator maximizes the conditional probability mass function
$\prod_{i \in S_1} p(y_i| \by_{-S}, \bX)$ 
with an $\ell_1$ regularization:
\begin{equation}\label{eq:define_theta_isl}
    \tilde{\btheta}= \argmin_{\bb} \left\{L_{S_1}(\bb) +\lambda_n \|\bb\|_1\right\},
\end{equation}
for some regularization parameter $\lambda_n>0$. Here, the log-pseudo-likelihood $L_{S_1}(\bb)$ is defined by
\begin{align}\label{eq:define_LS1}
    L_{S_1}(\bb)= -\frac{1}{|S_1|} \sum_{i \in S_1} \left\{ y_i(m_i(\by)+  \bX^\top_i \bb) -\log\cosh(m_i(\by)+  \bX^\top_i \bb)\right\},
\end{align}
where 
\begin{equation}\label{eq:mi_define}
    m_i(\by) =\sum_{\be: i \in \be} g_{\be}\by_{\be\setminus i}.
\end{equation} 

\noindent 
We will show the consistency of this estimator $\tilde \btheta$ in Proposition~\ref{thm:first_step}. As commonly seen in high-dimensional statistics, one needs to de-bias the $\ell_1$ regularized estimator to perform precise statistical inference~\cite{javanmard2018debiasing,javanmard2014hypothesis,zhang2014confidence}. To describe our debiasing method based on $\tilde \btheta$, we first transform the response variables $\overline y_i:= (y_i+1)/2 \in \{0,1\}$. Define the function 
\begin{align}\label{eq:define_f}
    f(x)= e^x/ (e^x+ e^{-x}).
\end{align}
Further define the quantities
\begin{equation}\label{eq:define_vi}
    v_i = m_i(\by)+ \bX^\top_i \btheta, \quad \tilde v_i= m_i(\by)+  \bX^\top_i \tilde \btheta.
\end{equation}
The rationale behind introducing the notation is $\overline y_i$, $i \in S_2$ are independent Bernoulli random variables given $\by_{-S}, \bX$ with 
\begin{equation}\label{eq:y_bern}
    \mathbb{P}(\overline y_i =1| \by_{-S}, \bX)= f(v_i).
\end{equation}
We consider the following bias-corrected estimator of the linear functional $\bc^\top\theta$, for some $c \in \mathbb R^d$: 
\begin{equation}\label{eq:define_hattheta_db}
    \bct= \bc^\top \tilde\btheta+ \frac{1}{|S_2|} \sum_{i \in S_2} \Big[w_i(\overline y_i - f(\tilde v_i)) \Big] \bu^\top \bX_i,
\end{equation}
where $\tilde\btheta$ and $\tilde v_i$ are defined by~\eqref{eq:define_theta_isl} and~\eqref{eq:define_vi} respectively. Here, $w_i$'s are suitable data-dependent weights and $\bu \in \mathbb R^d$ is a projection vector we will construct below.

Note that, using~\eqref{eq:y_bern}, we have for $i \in S_2$, $\text{Var}[\overline y_i| y_{-S}, \bX] = f(v_i)\big(1-f(v_i)\big)$. Define
\begin{equation}\label{eq:def_eps_i}
    \varepsilon_i= \overline {y}_i -f(v_i).
\end{equation}
Applying a Taylor series expansion of $f(\cdot)$ around $\tilde v_i$ to the summand in \eqref{eq:define_hattheta_db} yields

\begin{align}
    &\frac{1}{|S_2|} \sum_{i \in S_2} \Big[w_i \big(\overline y_i -f(\tilde v_i)\big)\Big] \bu^\top \boldsymbol X_i 
    \nonumber \\
    &= \frac{1}{|S_2|} \sum_{i \in S_2} w_i\Big[\big(\overline y_i-f(v_i)\big) + \big(f(v_i)-f(\tilde v_i)\big)\Big] \bu^\top \boldsymbol X_i \nonumber\\
    &= \frac{1}{|S_2|} \sum_{i \in S_2} w_i \epsilon_i \bu^\top\boldsymbol X_i + \frac{1}{|S_2|} \sum_{i \in S_2} w_i\Big[f'(\tilde v_i)\boldsymbol X_i^\top(\btheta - \tilde \btheta) + \underbrace{\frac{1}{2}f''(\tilde v_i + t\boldsymbol X_i^\top(\btheta- \tilde \btheta)) \big(\boldsymbol X_i^\top (\btheta-\tilde\btheta)\big)^2}_{=: \mathcal{R}_i}\Big] \bu^\top \bX_i \nonumber\\
    &= \frac{1}{|S_2|} \sum_{i \in S_2} w_i\epsilon_i\bu^\top\boldsymbol X_i + \sum_{i \in S_2} w_i\Big[f'(\tilde v_i)\boldsymbol X_i^\top(\boldsymbol{\theta-\hat\theta})\Big] \bu^\top \boldsymbol X_i + \frac{1}{|S_2|} \sum_{i \in S_2} w_i \mathcal{R}_i \bu^\top \boldsymbol X_i,
    \label{eq:TS1}
\end{align}

\noindent for some $t \in (0,1)$. Substituting \eqref{eq:TS1} into \eqref{eq:define_hattheta_db} yields
\begin{align}
    \bct- \bc^\top \btheta = \frac{1}{|S_2|} \sum_{i \in S_2} w_i \epsilon_i \bu^\top \boldsymbol X_i &+ \Bigg\{\bigg[\Big(\frac{1}{|S_2|} \sum_{i \in S_2} w_if'(\tilde v_i) \bu^\top \bX_i \bX_i^\top\Big) - \bc^\top\bigg](\btheta-\tilde\btheta)\Bigg\} \nonumber \\ &+ \frac{1}{|S_2|}\sum_{i \in S_2} w_i\mathcal{R}_i \bu^\top \bX_i. \label{eq:target_debias}
\end{align}

\noindent We will choose the weights $w_i$ and the projection vector $\bu$ in a way such that the first term is asymptotically normal and whose asymptotic variance is minimized; and the second and third summand will be negligible compared to the first quantity.
By the definition of $\varepsilon_i$ given by~\eqref{eq:def_eps_i}, we have
\begin{align*}
    \textrm{Var} \Bigg(
    \frac{1}{|S_2|} \sum_{i \in S_2} w_i \epsilon_i \bu^\top \boldsymbol X_i \Bigg| \bX, \by_{-S} \Bigg)= \frac{1}{|S_2|^2} \sum_{i \in S_2} w^2_i f(v_i) (1- f(v_i)) (\bu^\top \bX_i)^2
\end{align*}
Balancing the weights of each term of the above display with the second summand of the RHS, we will choose weights $w_i$ such that $w_i f'(\tilde v_i)= w^2_i f(\tilde v_i) (1- f(\tilde v_i))$, which implies that
\begin{equation}\label{eq:define_w}
    w_i= \frac{f'(\tilde v_i)}{f(\tilde v_i) (1- f(\tilde v_i))}.
\end{equation}
However, plugging the definition of $f$ defined by~\eqref{eq:define_f}, simple algebra yields that $w_i=2$. Given these constant weight, we finally choose the projection vector. Define the set
\begin{align}\label{eq:constrained set}
    \mathcal{A}:= &\Bigg\{ \bu \in \mathbb{R}^d: \Big\|\frac{2}{|S_2|}\sum_{i \in S_2} f'(\tilde v_i) \boldsymbol{X_i X_i^\top u} - \bc\Big\|_\infty \leq C_1 \|\bc\|_2\sqrt{\frac{\log d}{n}}, \nonumber \\
    & \Big|\frac{2}{|S_2|}\sum_{i \in S_2} f'(\tilde v_i) (\bu^\top \bX_i) (\bX^\top_i \bc) - \|\bc\|^2_2\Big| \leq C_2 \|\bc\|^2_2\sqrt{\frac{\log d}{n}}, \quad
    \max\limits_{i \in S_2} |\boldsymbol{X_i^\top u}| \leq C_3 \|\bc\|_2\sqrt{\log n} \Bigg\},
\end{align}
for some $C_1,C_2, C_3>0$. Our choice for the projection $\hat\bu$ is given by
\begin{align}
    \hat{\bu}:= \argmin_{ \bu \in \mathcal{A}}\frac{2}{|S_2|} \sum_{i \in S_2}  f'(\tilde  v_i) (\bu^\top\bX_i)^2, \label{eq:def-uhat} 
\end{align}
With this choice of $\hat \bu$ given by~\eqref{eq:def-uhat} and weights $w_i$ given by~\eqref{eq:define_w}, we have our bias-corrected estimator
\begin{equation}\label{eq:theta_db}
    \bct= \bc^\top \tilde\btheta+ \frac{2}{|S_2|} \sum_{i \in S_2} \Big[\overline y_i - f(\tilde v_i) \Big] \hat\bu^\top \bX_i,
\end{equation}
The use of projection vectors for bias-corrected estimation has been extensively studied in the context of generalized linear models with independent observations~\cite{ma_2020, cai_debiased_clt, guo2021inference}. Our main methodological contribution is to construct projection vectors $\hat\bu$ to settings with tensor network dependence among observations, and to establish asymptotic normality of linear functionals based on pseudo-likelihood estimators. Note that, unlike many prior works, we do not assume any $\ell_1$ or $\ell_2$-bounds on $\bc$, instead, our formulation in~\eqref{eq:constrained set} accommodates general $\bc$. If $\|\bc\|_1= O(1)$, the second constraint in~\eqref{eq:constrained set} is automatically satisfied by the first. As a consequence, to construct bias-corrected estimator for $\theta_j$ for any $j\in[d]$, we do not need the second constraint. This makes the computation of the projection vector faster.

\subsection{Hypothesis testing}\label{sec:testing}
We will show in Section~\ref{sec:theory} how the bias-corrected estimator $\bct$ is asymptotically normal under mild assumptions of covariates and standard sparsity assumptions. The main application of Theorem~\ref{thm:main} is to derive confidence intervals and do statistical hypothesis
tests for our high-dimensional network model.

\noindent \textbf{Confidence Interval:} 

To construct a confidence interval of $\bc^\top\btheta$, we need to estimate the asymptotic variance of $\bc^\top\btheta$. Note that the asymptotic variance of $\bct$, conditioned on $y_{-S}, y_{S_1} \bX$ is given by 
\begin{align}\label{eq:define_vj}
    V= \frac{1}{|S_2|^2} \sum_{i \in S_2} 4 f(v_i)\big(1-f(v_i)\big) (\hat \bu^\top \bX_i)^2.
    \end{align}
Given the definition of weights $w_i$ by~\eqref{eq:define_w}, we can estimate $V$ by plugging in the initial estimator $\tilde \btheta$ defined by~\eqref{eq:define_theta_isl}:
\begin{equation}\label{eq:estimated_var}
    \widehat V= \frac{1}{|S_2|^2} \sum_{i \in S_2} 4 f(\tilde v_i)\big(1-f( \tilde v_i)\big)(\hat \bu^\top \bX_i)^2,
\end{equation}
where $\tilde v_i$ is defined by~\eqref{eq:define_vi}. The natural $(1-\alpha)$-level confidence interval for $\bc^\top \btheta$ is now given by
\begin{equation}\label{eq:define_conf_int}
    \mathcal{C} :=  \mathcal{C}(\alpha) = \left[ \bct - z_{\alpha/2} \sqrt{\widehat V}, \bct+ z_{\alpha/2} \sqrt{\widehat{V}}\right],
\end{equation}
where $z_{\alpha}$ denotes the $\alpha$-th quantile of a standard normal distribution, and $\bct$ is defined by~\eqref{eq:theta_db}. Theoretical justification of this confidence interval is given by Proposition~\ref{prop:CI} under the standard sparsity assumption $s \ll \sqrt{n}$ (up to logarithmic factors). Note that, to construct an adaptive confidence interval without the knowledge of sparsity parameter $s$, one requires the condition $s \ll \sqrt{n}$ for independent data ($\bA(\bG_n) \equiv 0$) in both linear regression ~\cite{javanmard2018debiasing,van2014asymptotically} and logistic regression~\cite{cai_debiased_clt} problems.

An immediate consequence of the construction of the confidence interval $\mathcal{C}$ 
is that it can be used to conduct a test for the statistical hypothesis $H_0: \bc^\top \btheta = c^\star$. We choose the test statistic 
\begin{align}\label{eq:test_stat}
    T= \frac{(\bct -c^\star)}{\sqrt{\widehat V}}.
\end{align}
An $\alpha$-level test is the one which rejects $H_0$ if and only if $|T| \ge z_{\alpha/2}$. By selecting $\bc= \be_j$, we obtain inferential guarantees for each coordinate of the regression parameter $\theta_j$, $j\in [d]$. For such example, to denote the dependence on $j$, in the sequel we will denote the test statistic for testing $\theta_j$ by $T_j$ and corresponding estimated variance by $\widehat V_j$.

\noindent \textbf{Simultaneous Inference} 

Our method can be further used to perform multiple hypothesis testing for a number of regression coefficients simultaneously with desired control on Family-Wise Error Rate (FWER) or False Discovery Rate (FDR). In the standard multiple hypothesis problem, we are interested in
simultaneously testing the null hypotheses $H_{0j}: \theta_j = 0, \, j \in \mathcal{J}$ for any subset $\mathcal{J} \subseteq [d]$. Let $\{\hat{\theta}_j\}_{j \in \mathcal{J}}$ and
$\{\widehat{V}_j\}_{j \in \mathcal{J}}$ be the proposed bias-corrected estimators~\eqref{eq:theta_db} and their corresponding estimated variances~\eqref{eq:estimated_var} respectively for $\bc= \be_j$, $j \in \mathcal{J}$. 

To control FWER at a level $\alpha \in (0,1)$, the traditional method is to do a Bonferroni correction. This implies for any $j \in \mathcal{J}$, we reject the null hypothesis $H_{0j}$ whenever $|T_j| \ge z_{\alpha/(2|\mathcal{J}|)}$, where $T_j= \hat\theta_j/ \sqrt{\widehat{V}_j}$.
By the standard union bound, FWER can be controlled by
\[
\mathrm{FWER}
\leq |\mathcal{J}| \mathbb{P}_{H_{0j}}\!\left(|T_{j}| \geq z_{\alpha/(2|\mathcal{J}|)}\right)
\leq \alpha.
\]
However, if $|\mathcal{J}|$ is too large, such a procedure might yield low power. Therefore one might use the bootstrap-based alternative cutoffs~\cite{chernozhukov2017central, dezeure2017high} popular in high-dimensional statistics.Since Bonferroni procedure is more conservative, it is often preferable to consider FDR control instead. A standard method in this regard is the modified Benjamini-Hochberg (BH) algorithm~\cite{benjamini1995controlling} 
which, for a carefully chosen cutoff $\kappa$, reject all the null hypothesis $H_{0j}$ such that $|T_j| \ge \kappa$. If $\mathcal{J}_0$ denote the subset of all \textit{true} null hypothesis in the set $\mathcal{J}$, we have 
\[
\mathrm{FDR}(\kappa) = \mathbb{E}\!\left[
\frac{\sum_{j \in \mathcal{J}_0} 1\{|T_{j}| \geq \kappa\}}
{\max\{\sum_{j \in \mathcal{J}} 1\{|T_{j}| \geq \kappa\}, 1\}}
\right].
\]
Here, $\sum_{j \in \mathcal{J}} 1\{|T_{j}| \geq \kappa\}$ denote the total number of rejected hypothesis. By assuming number of true alternatives are sparse, the proportion of nulls falsely
rejected among all the true nulls at the threshold level $\kappa$, 
$\frac{1}{|\mathcal{J}_0|}\sum_{j \in \mathcal{J}_0} 1\{|T_{j}| \geq \kappa\}$,  can be approximated by standard normal tail
$2 - 2\Phi(\kappa)$. This approximation implies that for $0 < \alpha <1$, the choice of cutoff $\hat\kappa$ is given by
\[
\hat\kappa:= \inf \left\{  0 \leq \kappa \leq \sqrt{2 \log |\mathcal{J}| - 2 \log \log |\mathcal{J}|} :
\frac{|J|\{2 - 2\Phi(\kappa)\}}
{\max\{\sum_{j \in \mathcal{J}} 1(|T_{j}| \geq \kappa), 1\}} \leq \alpha\right\}
\]
The range of $\kappa$ and the proof that the above method controls FDR at level-$\alpha$ follows the same steps as ~\cite{liu2013gaussian, ma_2020} as $n, |\mathcal{J}| \rightarrow \infty$.

\subsection{Inference for quadratic functionals} \label{sec:quadratic}
Our method developed in Section~\ref{sec:lin_func} can be further extended for statistical inference of quadratic functional $Q= \btheta^\top \bM \btheta$, for some symmetric positive definite $\bM \in \mathbb{R}^{d \times d}$. The common choice of $\bM$ are $\mathbf{I}$ or $\bSigma= \textrm{Var}(\bX_i)$. 
The quadratic form has been used often in applications related to heritability estimation in Genome-Wide Association Studies (GWAS)~\cite{manolio2009finding,yang2015genetic}. Often in literature, the estimation procedure of heritability relies on the
classical asymptotics that does not take into account the high-dimensionality of the SNPs compared to the sample sizes~\cite{bulik2015atlas} or consider continuous traits \cite{guo2019optimal,zhao2023cross} or does not consider dependence among phenotypes \cite{ma2024statistical}. We address these gaps by providing inference under quadratic functionals under dependent observations.

Note that our first-step estimator $\tilde \btheta$ has the following bias:
\begin{equation}\label{eq:quad_split}
    \tilde \btheta^\top \bM \tilde \btheta- Q= 2 \tilde \btheta^\top \bM (\tilde \btheta- \btheta) - (\tilde \btheta- \btheta)^\top \bM (\tilde \btheta- \btheta).
\end{equation}

The second term in the RHS of the above display is of smaller order compared to the first term, which causes the choice of estimator $\tilde \btheta^\top \bM \tilde \btheta$ biased. However, the first summand can be expressed as the linear form $\bc^\top (\tilde\btheta -\btheta)$ for $\bc^\top = 2\tilde\btheta^\top \bM$. Hence, by our construction of bias-corrected estimator of linear functional, we choose the projection vector $\hat\bu_M= \hat\bu_M (\tilde\btheta)$ given by~\eqref{eq:def-uhat}. Note that, here we choose the constrained set $\mathcal{A}= \mathcal{A}(\bM)$ 
will be defined as follows:
\begin{align}\label{eq:constrained set_quad}
    \mathcal{A}:= &\Bigg\{ \bu \in \mathbb{R}^d: \Big\|\frac{2}{|S_2|}\sum_{i \in S_2} f'(\tilde v_i) \boldsymbol{X_i X_i^\top u} - \bM \tilde\btheta\Big\|_\infty \leq C_1 \|\bM \tilde\btheta\|_2\sqrt{\frac{\log d}{n}}, \nonumber \\
    & \Big|\frac{2}{|S_2|}\sum_{i \in S_2} f'(\tilde v_i) (\tilde\btheta^\top \bM \bX_i) (\bu^\top \bX_i) - \|\bM \tilde\btheta\|^2_2\Big| \leq C_2 \|\bM \tilde\btheta\|^2_2\sqrt{\frac{\log d}{n}}, \nonumber \\
    & \max\limits_{i \in S_2} |\boldsymbol{X_i^\top u}| \leq C_3 \|\bM \tilde\btheta\|_2\sqrt{\log n} \Bigg\},
\end{align}
for some constants $C_1,C_2,C_3>0$. Given such a projection vector $\hat\bu_M$, we define the estimator of $Q$ as $\widehat Q =\max (\tilde Q, 0)$ where $\tilde Q$ is given by:
\begin{equation*}
    \tilde Q= \tilde \btheta^\top \bM \tilde \btheta +  \frac{4}{|S_2|} \sum_{i \in S_2} \Big[\bar{y_i} - f(\tilde v_i) \Big] \hat\bu^\top_M \bX_i, 
\end{equation*}
where recall the definition of the function $f$ from~\eqref{eq:define_f}. Since $\bM$ is positive definite matrix, truncating at $0$ yields a better estimator than using $\tilde Q$. Next, we estimate the variance of the estimator (conditional on $\by_{-S}, \by_{S_1}, \bX$) by
\begin{equation}
    \widehat V_{\bM}= \frac{16}{|S_2|^2} \sum_{i \in S_2} f(\tilde v_i)\big(1-f( \tilde v_i)\big)(\hat \bu^\top_M \bX_i)^2 +\frac{1}{n}.
\end{equation}
Here we have added the quantity $1/n$ as the upper bound of the second summand in~\eqref{eq:quad_split}; similar modification is common in classical i.i.d. setup~\cite{ma2024statistical}. As a consequence, the natural $(1-\alpha)$-level confidence interval can be given by
\begin{align}\label{eq:quad_ci}
    \mathcal{C}_{\bM}= \left[ \max\Big(\widehat Q- z_{1-\alpha/2} \sqrt{\widehat V_{\bM}},0 \Big), \widehat Q + z_{1-\alpha/2} \sqrt{\widehat V_{\bM}} \right]
\end{align}

We want to highlight that throughout the section we have assumed that the matrix $\bM$ is known to the statistician. However, there might be applications where $\bM$ is unknown. For example, in population genetic, one needs to estimate heritability given by $Q= \btheta^\top \bSigma \btheta$, where the population covariance matrix $\bSigma$ is unknown to a geneticist. In that case, we can plug-in an estimator $\widehat\bSigma$ in our procedure to perform statistical inference; similar idea has been explored for i.i.d. datasets recently~\cite{guo2019optimal,ma2024statistical}. In the semi-supervised setting~\cite{tony2020semisupervised}, where additional unlabeled covariates are available, one can use them to construct an estimator $\widehat\bSigma$.

\subsection{Theoretical guarantees}\label{sec:theory}

In this section, we provide theoretical justification behind our proposed estimator $\bct$ 
defined by~\eqref{eq:theta_db}. We begin by stating our assumptions on the interaction tensors $\bA_n$, 
the regression parameter $\btheta$ and the covariate distribution $\mathbb{P}_X$. For any vertex $i \in [n]$ of the hypergraph $\bG_n$, we define its neighbors by 
\begin{equation}\label{eq:defn_nbr}
\mathcal{N}_i := \left\{j \in [n]: \exists \be \in E \quad \text{such that} \quad i,j \in \be \right\}
\end{equation}
\begin{ass}\label{assn:graph}
    The hypergraph $\bG_n$ satisfies bounded degree condition 
    $\sup_n \max_{i \in [n]} \sum_{\be: i \in \be} g_{\be} \le 1$ and $\sup_n \max_{i \in [n]} |\mathcal{N}_i|= \Delta <\infty$.
\end{ass}
\begin{ass}\label{assn:theta}
    The regression coefficients $\btheta \in \mathbb{R}^d$ satisfies $\|\btheta\|_0=s$, $\|\btheta\|_2 \le C$ for some $C>0$.
\end{ass}
\begin{ass}\label{assn:covariates}
    The covariates $\bX_i \in \mathbb{R}^d$ are i.i.d. sub-Gaussian vectors with bounded sub-Gaussian norm $C_1>0$. Define by $\bSigma= \mathbb{E}(\bX_1 \bX^\top_1) \in \mathbb{R}^{d \times d}$ the covariance matrix of the covariates. There exists $\kappa>0$ independent of $n$ and $d$ such that 
    \[
    \frac{1}{\kappa} \le \lambda_{\min}(\bSigma) \le \lambda_{\max}(\bSigma) \le \kappa.
    \]
\end{ass}
Note that, the assumptions on the regression parameters and covariates are standard in high-dimensional statistics literature~\cite{buhlmann2011statistics}. In Assumption~\ref{assn:graph}, we assume the underlying tensor to be sparse. 
This is required to obtain the strong-independent set $S(\mathbf{G}_n)$ large. Indeed if the maximum size of the neighborhoods $\max_i|\mathcal{N}_i| \le \Delta$, then the graph contains a strong-independent set of size at least $\lfloor n/(\Delta+1) \rfloor$. Further, one can efficiently find such a strong-independent set of size $\Omega(n)$ by a greedy algorithm. Therefore, we will assume throughout that the set used to construct our estimator satisfies 
\begin{equation}\label{eq:size_set}
    |S(\mathbf{G}_n)|= \Omega(n).
\end{equation}
We are now ready to state our first theorem.
\begin{theorem}\label{thm:main}
     Suppose the Assumptions~\ref{assn:graph}-\ref{assn:covariates} hold and $V$ is defined by~\eqref{eq:define_vj}. If $s = o\Big(\frac{\sqrt n}{\sqrt{\log n} \log d}\Big)$ and $\lambda_n= C \sqrt{\log d/n}$ for some $C>0$, then 
     we have
     \begin{equation}
         \frac{(\bct - \bc^\top \btheta)}{\sqrt{\widehat V}} \xrightarrow{D} Z,
     \end{equation}
     where $Z \sim N(0,1)$ and $\bct$ is defined by~\eqref{eq:theta_db}.
\end{theorem}
The result characterizes the precise asymptotic distribution of $\bct$. The asymptotic variance involves the quantity $V$ which is generally unknown to the statistician. As we will show in Proposition~\ref{prop:CI} below, replacing $V$ with its natural estimator $\widehat V$ defined by~\eqref{eq:estimated_var} yields a valid confidence interval. Theorem~\ref{thm:main} contributes to the long line of literature of inference in generalized linear models~\cite{van2014asymptotically,zhang2014confidence,cai_debiased_clt, guo2021inference,jankova2016confidence}. However, to our knowledge, ours is the first theorem to provide inferential guarantees under network dependence models in high dimension. Further we note that many of the existing methods of statistical inference requires boundedness on the individual probability, i.e., assumes there exists $\mu \in (0, 0.5)$ such that $\mathbb{P}(y_i=1|\bX_i) \in (\mu,1-\mu)$~\cite{ning2017,shi2021statistical}. However, the current work does not posit such conditions. Given our tensor dependence among the subjects, the individual probabilities can get arbitrarily close to $0$ or $1$.

Our next result provides the asymptotic coverage guarantee of the confidence interval $\mathcal{C}$ defined by~\ref{eq:define_conf_int}, along with an upper bound on the length of the confidence interval.
\begin{prop}\label{prop:CI}
    Suppose the assumptions of Theorem~\ref{thm:main} hold. Then for any 
    $0<\alpha <1$, we have
    \begin{equation*}
        \lim\limits_{n \rightarrow \infty} \inf_{\btheta} \mathbb{P}(\bc^\top \btheta \in \mathcal{C}) \ge 1- \alpha,
    \end{equation*}
    where $\mathcal{C}$ is defined by~\eqref{eq:define_conf_int}. Further, defining $L(\mathcal{C})$ to be the length of the confidence interval, we have for any $\delta>0$,
    \begin{align*}
       \lim\limits_{n \rightarrow \infty} \mathbb{P}_{\btheta}\Big(L(\mathcal{C}) \ge (1+\delta) 2 z_{\alpha/2}\sqrt{V}\Big)=0,
    \end{align*}
    where $V$ is defined by~\eqref{eq:define_vj}.
\end{prop}
Similar to traditional i.i.d. settings~\cite{belloni2016post,van2014asymptotically}, the confidence interval proposed here is agnostic of true sparsity level and remains valid for sparsity level $o(\sqrt{n})$, ignoring the logarithmic factors. Further, note that if $\|\bc\|_2=O(1)$, then with high probability the length of the confidence interval is $O(n^{-1/2})$, matching with the optimal interval length in the classical i.i.d. response setting. We note that for quadratic functionals, the theoretical justification of the confidence interval $\mathcal{C}_M$ defined by~\eqref{eq:quad_ci} holds by the same argument.

A natural consequence of our result is the validity of hypothesis testing using the test-statistic constructed in Section~\ref{sec:testing}. Consider the hypothesis testing problem
    \begin{align}\label{eq:hypo_Test}
        H_0: \bc^\top \btheta \le 0, \quad \text{vs.} \quad H_1: \bc^\top\btheta = \upsilon n^{-1/2},
    \end{align}
for some $\upsilon >0$. We consider the test which rejects $H_0$ if $\bct \ge z_{\alpha} \sqrt{\widehat{V}}$, where $\bct$ and $\sqrt{\widehat{V}}$ are defined by~\eqref{eq:theta_db} and~\eqref{eq:estimated_var} respectively. We have the following theoretical guarantee for this testing procedure as an immediate consequence of Theorem~\ref{thm:main} and Proposition~\ref{prop:CI}.
    
\begin{prop}\label{prop:testing}
    Suppose the assumptions of Theorem~\ref{thm:main} hold. Consider the hypothesis testing problem given by~\eqref{eq:hypo_Test} for some $0 < \alpha <1$, we have $\lim_{n\rightarrow \infty} \mathbb{P}_{\btheta}(\bct \ge z_{\alpha} \sqrt{\widehat{V}}) \le \alpha$ for any $\btheta \in H_0$. Further, for any $\btheta \in H_1$, we have
    \begin{align*}
        \lim\limits_{n \rightarrow \infty} \left|\mathbb{P}_{\btheta}\left(\bct \ge z_{\alpha} \sqrt{\widehat{V}}\right)- \left\{1- \Phi\left(z_{\alpha}- \frac{\upsilon}{\sqrt{n V}}\right)\right\}\right|= 0.
    \end{align*}
\end{prop}
Proposition~\ref{prop:testing} implies that our testing procedure has the desired type-I error. For any $\upsilon>0$, 
the result also provides the precise asymptotic power of the testing procedure. Note that by the definition of $V$ in~\eqref{eq:define_vj}, we have $\sqrt{n V}/ \|\bc\|_2$ both upper and lower bounded. Therefore if we have $\upsilon/ \|\bc\|_2 \rightarrow \infty$, the asymptotic power will be exactly $1$.

The performance of the bias-corrected estimator hinges on the estimation of the regularized pseudo-likelihood based estimator $\tilde \btheta$ which is summarized in the following result.

\begin{prop}\label{thm:first_step}
    Suppose that assumptions of Theorem~\ref{thm:main} hold, and $\tilde\btheta$ is defined by~\eqref{eq:define_theta_isl}. If we select the regularization parameter $\lambda_n = C\sqrt{\log d/n}$ in \eqref{eq:define_theta_isl} for some $C>0$,
\begin{align*}
\|\tilde\btheta- \btheta\|_2 = O\bigg(\sqrt{\frac{s\log d}{n}}\bigg) \; \; \text{and} \; \; \|\tilde\btheta- \btheta\|_1 = O\bigg(s\sqrt{\frac{\log d}{n}}\bigg),
\end{align*}
with probability $1-o(1)$ as $n \to \infty$ and $d \to \infty$ such that $s\sqrt{\frac{\log d}{n}} = o(1)$.
\end{prop}
Our result achieves the same estimation error rate as the Lasso estimator in classical high-dimensional settings~\citep{buhlmann2011statistics}. The key distinctions between Proposition~\ref{thm:first_step} and~\cite[Theorem 2.9]{mukherjee2024logistic} lie in the model and estimation framework: our analysis accommodates general hypergraph-induced dependence and sub-Gaussian covariates, thereby broadening the applicability of the result. Moreover, to enable bias correction, we employ a maximum pseudolikelihood estimator constructed on a strong independent set. These features make our proof different from that of~\cite[Theorem 2.9]{mukherjee2024logistic}, and we provide the proof of Proposition~\ref{thm:first_step} in the Appendix.

\section{Simulations}\label{sec:sims}

In this section, we complement our theoretical findings with numerical examples from three simulations. In each simulation, the underlying network is given by a graph $\bG_n$ whose connectivity tensor is its adjacency matrix $\bA(\bG_n)$. The first example considers $\bG_n$ to be the 2-dimensional lattice on $\sqrt{n} \times \sqrt{n}$ vertices, with $\bA_n= \frac{1}{4} \bA(\bG_n)$ to be the scaled adjacency matrix of $\bG_n$. We observe $(y,\bX)$ such that 
\begin{equation}\label{eq:def_ising}
    \mathbb{P}_{\theta}(\by|\bX) \propto \exp \left(\beta \by^\top \bA_n \by + \sum_{i=1}^{n} y_i \bX^\top_i \btheta\right).
\end{equation}
Since the maximum degree of this graph is $4$, it satisfies our Assumption~\ref{assn:graph} for any $\beta$. The parameter $\beta>0$, known as the ``inverse temperature", captures the strength of network effects. The regression parameter $\btheta \in \mathbb{R}^d$ is $s$- sparse, and we set $\theta_1=\ldots=\theta_s=1$. Our observed covariates are drawn as $\bX_i \sim N(0, \bSigma)$, where $\bSigma$ is the autocovariance matrix with $\Sigma_{ij}= 0.2^{|i-j|}$. To generate the observation $\by$, we use Gibbs sampling with $2000$ iterations, following the procedure given in ~\cite{mukherjee2024logistic}. We consider the network induced by the $40 \times 40$ lattice, i.e. $n=1600$, with $d=100$ and $s=5$. For varying network strength, we choose $\beta \in \{0.1,0.15,0.2,0.25,0.3\}$. We select a non-zero component (here, we selected $\theta_2$) and construct a $95\%$ confidence interval as described by our method~\eqref{eq:define_conf_int}. As no direct comparison methods exist for this problem, we benchmark against the confidence interval for high-dimensional logistic regression provided by~\cite{cai_debiased_clt}, which does not consider network effects among the responses. We report the coverage proportion of both confidence intervals over $100$ iterations in Table~\ref{tab:tab1}. Ignoring the network effects leads to severe under-coverage, as reflected in the third column. One might wonder whether the coverage of our method simply arises from wider intervals; this is not the case. For $\beta=0.2$ in the above setup, the maximum length of the confidence intervals is $0.59$ with the median length being $0.47$. We note that the lengths are similar to the ones obtained by the baseline method ignoring network effects. Theoretical bounds on interval length were established earlier in Proposition~\ref{prop:CI}.

\begin{table}[ht]
 \begin{center}
	\begin{small}
		\begin{tabular}{ p{4cm} p{4cm}  p{4 cm} }
			\toprule
			Network strength & Proposed method & Baseline method\\
			\midrule
0.1   &  0.97  &  0.26  \\
0.15   &  0.98  &  0.27  \\
0.2  &  0.98  &   0.24 \\
0.25   &  0.98 &  0.22  \\
0.3   &  0.99  &  0.22  \\			
			\bottomrule
		\end{tabular}
	\end{small}
    \caption {Coverage proportion of a randomly chosen non-zero coordinate of $\btheta$ for varying $\beta$. We construct a $95\%$ confidence interval based on~\eqref{eq:define_conf_int}. Note the stark difference in coverage if the confidence interval for standard high-dimensional logistic regression without network effects is used.} \label{tab:tab1}
  \end{center}
\end{table}

Our second example examines the validity of our confidence interval across different network structures. We set $\bG_n$ to be uniformly chosen $\Delta$- regular graph with varying $\Delta=\{4,\ldots, 8\}$. Define $\bA_n= \frac{1}{\Delta}\bA(\bG_n)$. We set $\beta=0.2$, with all other parameters as in the previous example. We construct $95\%$ confidence interval using our proposed method and report the coverage proportion in Table~\ref{tab:tab2} along with coverage proportion of baseline method. Note that, we have constructed the independent set using greedy algorithm to ensure it has size $\lfloor n\Delta \rfloor$. The overall picture mirrors that of the first example: our method achieves the desired coverage, while ignoring peer effects yields poor uncertainty quantification.

\begin{table}[ht]
 \begin{center}
	\begin{small}
		\begin{tabular}{ p{4cm} p{4cm}  p{4 cm} }
			\toprule
			Network-degree & {\small Proposed method} & {\small Baseline method}\\
			\midrule
4   &  0.97  &  0.18  \\
5   &  0.94  &  0.17  \\
6  &  0.94  &   0.27 \\
7   &  0.94 &  0.20  \\
8   &  0.95  &  0.24  \\			
			\bottomrule
		\end{tabular}
	\end{small}
    \caption {Coverage proportion of a randomly chosen non-zero coordinate of $\btheta$ for different underlying network. We consider a random regular graph with all vertices having degree $\Delta$. We vary $\Delta$, as denoted in the first column. As before, we notice that ignoring network dependence has severe under-coverage.} \label{tab:tab2}
  \end{center}
\end{table}

Our final example studies the effect of $d$, the dimension of $\btheta$. We set $\bG_n$ to be the two-dimensional lattice of the first example with $\beta=0.2$. Here, we vary $d= \{100,125,150, 175, 200\}$, with the first $s$ coordinates of $\btheta$ to be non-zero throughout. We select a non-zero coordinate of $\btheta$ and construct $95\%$ confidence interval, with all other parameters as in the first example. Our coverage of the interval is reported in Table~\ref{tab:tab3} where proportion is constructed over $100$ iterations of the simulation. Our method consistently achieves the desired coverage, whereas ignoring the peer effect leads to increasingly poor coverage as $d$ grows.

\begin{table}[ht]
 \begin{center}
	\begin{small}
		\begin{tabular}{ p{4cm} p{4cm}  p{4 cm} }
			\toprule
			d & {\small Proposed method} & {\small Baseline method}\\
			\midrule
100   &  0.98  &  0.24  \\
125   &  0.98  &  0.16  \\
150   &  0.97  &  0.14  \\
175   &  0.97  &  0.14  \\
200  &  0.96  &   0.19 \\		
			\bottomrule
		\end{tabular}
	\end{small}
    \caption {Coverage proportion of a randomly chosen non-zero coordinate of $\btheta$ for varying dimension of regression parameter, as denoted in the first column. We set $2$-dimensional lattice as the underlying network. The result shows that our proposed method yields valid coverage throughout each setup.} \label{tab:tab3}
  \end{center}
\end{table}

\section{Discussion and future directions}\label{sec:discussion}

Our work develops a statistical inference framework for logistic regression in the presence of network dependence, where the dependence structure is induced by a hypergraph. Moving beyond consistent estimation problems in Markov random fields with pairwise (quadratic) interactions, we establish inferential guarantees for both linear and quadratic functionals in tensor MRFs. Our results contribute to the broader literature on inference in dependent data settings and suggest several promising directions for future research. We highlight three such directions below.

First, our current framework focuses on sparse networks, which facilitates the identification of strong independent sets. An important open question is how to extend these ideas to dense graphs. Although estimation and testing problems in dense Ising models have been studied extensively in the past decade~\cite{bhattacharya2025sharp,mukherjee2018global,mukherjee2022testing}, these works typically do not incorporate high-dimensional covariates. Developing inference procedures that simultaneously account for dense dependence structures and high-dimensional covariates remains an interesting challenge. Second, the estimation of causal effects under network interference has attracted considerable recent attention~\cite{tchetgen2021auto, bhattacharya2025causal}. These works, however, focus on settings without high-dimensional covariates. Our proposed framework naturally suggests a way to incorporate high-dimensional predictors and may provide inferential guarantees for both direct and spillover effects in such causal inference problems. Extending our bias-correction and projection techniques to this setting would be an important contribution. Finally, we have restricted our analysis to binary data. It would be of substantial interest to extend our framework to general network-assisted generalized linear models. In particular, it remains to investigate how the construction of the projection vector (cf.~\eqref{eq:def-uhat}) can be generalized to accommodate more general responses and link functions beyond the logistic. Such extensions would further broaden the applicability of our approach to a wide range of network-dependent data structures.

\bibliography{references}

\section*{Appendix}
This appendix contains the proofs of our theoretical results. We will first prove our main result, i.e. asymptotic normality of debiased estimator.

\subsection*{Proof of Theorem~\ref{thm:main}}
In this section, we will prove the asymptotic distribution of the bias-corrected estimator $\bct$. Following the decomposition~\eqref{eq:target_debias}, and recalling that $w_i=2$ using~\eqref{eq:define_w}, we obtain that our estimator satisfies 
\begin{align}\label{eq:define_an_bn}
    &\bct - \bc^\top \btheta = \underbrace{\frac{2}{|S_2|} \sum_{i \in S_2} \epsilon_i \hat\bu^\top \boldsymbol X_i}_{=:\mathcal{T}_1} \nonumber \\ &+ \underbrace{\Bigg\{\bigg[\Big(\frac{2}{|S_2|} \sum_{i \in S_2} f'(\tilde v_i) \hat\bu^\top \bX_i \bX_i^\top\Big) - \bc^\top\bigg](\btheta-\tilde\btheta)\Bigg\} + \frac{2}{|S_2|}\sum_{i \in S_2} \mathcal{R}_i \hat\bu^\top \bX_i}_{=:\mathcal{T}_2}. 
\end{align}
We will show that $\mathcal{T}_1$ converges weakly to a normal distribution with appropriate variance and $\mathcal{T}_2$ is small in probability under our sparsity assumptions on $\btheta$. The precise claims are given as follows:
\begin{enumerate}
    \item[(i)] $\frac{\mathcal{T}_1}{\sqrt{\widehat{V}}}\xrightarrow{D} Z$, where $Z \sim N(0,1)$ and $\widehat V$ is defined by~\eqref{eq:estimated_var}.
    \item[(ii)] The projection vector $\hat\bu$ defined by~\eqref{eq:def-uhat} satisfies
    \begin{align*}
        &\bigg\|\Big(\frac{2}{|S_2|} \sum_{i \in S_2} f'(\tilde v_i) \hat\bu^\top \bX_i \bX_i^\top\Big) - \bc^\top\bigg\|_{\infty} =O_{\mathbb{P}}(\|\bc\|_2\lambda_n), \\ 
        &\bigg|\Big(\frac{2}{|S_2|} \sum_{i \in S_2} f'(\tilde v_i) \hat\bu^\top \bX_i \bX_i^\top \bc\Big) - \|\bc\|^2_2\bigg| =O_{\mathbb{P}}(\|\bc\|^2_2\lambda_n), \max_{i \in S_2} |\hat\bu^\top \bX_i| = O_{\mathbb{P}}(\|\bc\|_2 \sqrt{\log n}),
    \end{align*}
    where $\tilde v_i$ is defined by~\eqref{eq:define_vi}.
    \item[(iii)]  \[ \left|\frac{1}{|S_2|}\sum_{i \in S_2} \mathcal{R}_i \hat\bu^\top \bX_i \right| =O_{\mathbb{P}}\left( \|\bc\|_2\frac{s \log d \sqrt{\log n}}{n}\right).\]
\end{enumerate}
Assuming the three claims are true, we immediately have, using $\lambda_n= O(\sqrt{\frac{\log d}{n}})$,
\begin{align*}
    |\mathcal{T}_2| &\le \bigg\|\Big(\frac{2}{|S_2|} \sum_{i \in S_2} f'(\tilde v_i) \hat\bu^\top \bX_i \bX_i^\top\Big) - \bc^\top\bigg\|_{\infty} \|\tilde \btheta- \btheta\|_1 + O_{\mathbb{P}}\left( \|\bc\|_2\frac{s \log d \sqrt{\log n}}{n}\right) \\
    &= O_{\mathbb{P}}(\|\bc\|_2 s \lambda^2_n)+ O_{\mathbb{P}}\left(\|\bc\|_2 \frac{s \log d \sqrt{\log n}}{n}\right) = O_{\mathbb{P}}\left(\|\bc\|_2 \frac{s \log d \sqrt{\log n}}{n}\right) \\
    &=o_{\mathbb{P}}\left(\frac{\|\bc\|_2}{\sqrt{n}}\right)  =o_{\mathbb{P}}\left(\sqrt{\widehat V}\right),
\end{align*}
where the first equality above follows from the bound on $\|\tilde\btheta- \btheta\|_1$ given by Proposition~\ref{thm:first_step}, the third equality by our sparsity assumption in the statement of the theorem, and the final equality by ~\eqref{eq:hatv_consistency} and~\eqref{eq:v_lower_bound} . Hence, we have the desired conclusion. Now we turn to the proof of three claims above.

\noindent \textbf{Proof of $(i)$:}

Define $\zeta_i = 2 V^{-1/2} \epsilon_i \hat{\boldsymbol u}^\top \boldsymbol X_i$, $V$ is defined by~\eqref{eq:define_vj}. Conditional on $\bX_{-S_2}$, the random variables $\{\zeta_i\}_{i \in S_2}$ are independent mean $0$ , and
\begin{align*}
    s_n^2 &:= \sum_{i \in S_2} \text{Var}[\zeta_i] = \sum_{i \in S_2} \frac{\big(2\hat{\boldsymbol u}^\top \boldsymbol X_i\big)^2}{V}f(v_i)\big(1-f(v_i)\big) \\
    &= \frac{1}{V}  \Big[\sum_{i \in S_2}4 (\hat\bu^\top \boldsymbol X_i)^2 f(v_i)\big(1-f(v_i)\big)\Big] = |S_2|^2,
\end{align*}
where $V$ is defined by~\eqref{eq:define_vj}. To obtain the asymptotic normality, it suffices now to show the Lindeberg condition for the sequence $\{\zeta_i\}$, namely that for any $\gamma>0$,
\begin{align*}
    \lim\limits_{n \to \infty} \frac{1}{|S_2|^2} \sum_{i \in S_2} \mathbb{E} \big[\zeta_i^2 \mathbbm{1}_{|\zeta_i| > \gamma |S_2|}\big]=0.
\end{align*}
In fact, we will show the following: there exists $C_1 > 0$ such that $V \ge C_1  \|\bc\|^2_2/n$. Then, we would have
\begin{align*}
    \max_i|\zeta_i| \lesssim \frac{\sqrt{n}}{\|\bc\|_2}\max_i|\hat \bu^\top \bX_i| = O\Big(\sqrt{n \log n}\Big) = o(n),
\end{align*}
where the first inequality uses $|\varepsilon_i| \le 1$,  and the final inequality uses that $\hat \bu \in \mathcal{A}$ defined by~\eqref{eq:constrained set}. This implies the Lindeberg condition mentioned above immediately proving our claim. 

The remaining part of this subsection is devoted to providing a lower bound for $V$ which holds with high probability. To this end, define
\begin{align*}
    \hat{\boldsymbol \Gamma}:= \frac{2}{|S_2|} \sum_{i \in S_2} f'(\tilde v_i) \bX_i \bX_i^\top, \quad 
    \tilde V := \frac{1}{|S_2|}\hat{\boldsymbol u}^\top\hat{\boldsymbol\Gamma}\hat{\boldsymbol u}.
\end{align*}
With this definition, we plan on showing that $\frac{\tilde V}{V} \stackrel{\mathbb{P}}{\to} 1$. Recall that $f'(x)= 2 f(x)(1-f(x))$ from our definition~\eqref{eq:define_w}. Therefore, it suffices to show that
\begin{align*}
    \max_{i \in S_2}\left|\frac{f(\tilde v_i)\big(1-f(\tilde v_i)\big)}{f(v_i)\big(1-f(v_i)\big)}-1\right| \stackrel{\mathbb{P}}{\to} 0.
\end{align*}
Given consistency of $\tilde \btheta$, the proof of this claim in the above display is included in the proof of Proposition~\ref{prop:CI}. The consistency of $\tilde \btheta$ is shown already in Proposition~\ref{thm:first_step}. Hence to provide a lower bound for $V$, it is enough to provide a lower bound for $\tilde V$.

Turning to yielding a lower bound on $\hat\bu^\top \hat {\boldsymbol{\Gamma}} \hat\bu$, define the quantity
\begin{align}\label{eq:utilde}
    \tilde{\boldsymbol u} := \argmin_{\boldsymbol u \in \mathbb R^d} \boldsymbol u^\top\hat{\boldsymbol \Gamma} \boldsymbol u, \qquad
    \text{subject to} \qquad \Big|\bc^\top \hat{\boldsymbol \Gamma}\boldsymbol u - \|\bc\|^2_2\Big| \leq C\|\bc\|^2_2\lambda_n.
\end{align}
for some large $C>0$. Since our projection vector $\hat \bu$ defined by~\eqref{eq:def-uhat} lies in the feasible set defined in the above display, we have for any $t \geq 0$,
\begin{align}\label{eq:lemma-3-1}
    \hat{\boldsymbol u}^\top\hat{\boldsymbol \Gamma} \hat{\boldsymbol u} \geq \tilde{\boldsymbol u}^\top\hat{\boldsymbol \Gamma} \tilde{\boldsymbol u} \geq \tilde{\boldsymbol u}^\top\hat{\boldsymbol \Gamma} \tilde{\boldsymbol u} + t\big((1-\lambda_n)\|\bc\|^2_2 - \bc^\top\hat{\boldsymbol\Gamma}\tilde{\boldsymbol u}\big),
\end{align}
\noindent where the first inequality follows from the definition of $\tilde{\boldsymbol u}$, and the second inequality follows from the constraint \eqref{eq:utilde}. Therefore we have
\begin{align}\label{eq:lemma-3-2}
    \tilde{\boldsymbol u}^\top\hat{\boldsymbol \Gamma} \tilde{\boldsymbol u} + t\big((1-C\lambda_n)\|\bc\|^2_2 - \boldsymbol c^\top\hat{\boldsymbol\Gamma}\tilde{\boldsymbol u}\big) &\geq \min\limits_{u \in \mathbb R^d} \boldsymbol u^\top\hat{\boldsymbol \Gamma} \boldsymbol u + t\big((1-C\lambda_n)\|\bc\|^2_2 - \boldsymbol c^\top\hat{\boldsymbol\Gamma}\boldsymbol u\big) \nonumber \\
    &= -\frac{t^2}{4}\boldsymbol c^\top\hat{\boldsymbol \Gamma}\boldsymbol c + t(1-C\lambda_n)\|\bc\|^2_2.
\end{align}

\noindent Combining \eqref{eq:lemma-3-1} and \eqref{eq:lemma-3-2}, we have that $\hat{\boldsymbol u}^\top\hat{\boldsymbol \Gamma} \hat{\boldsymbol u} \geq -\frac{t^2}{4}\bc^\top\hat{\boldsymbol \Gamma}\bc + t(1- C \lambda_n)\|\bc\|^2_2$ for all $t \geq 0$, and therefore
\begin{align*}
    \hat{\boldsymbol u}^\top\hat{\boldsymbol \Gamma} \hat{\boldsymbol u} \geq \max\limits_{t \geq 0} \Big[-\frac{t^2}{4}\bc^\top\hat{\boldsymbol \Gamma}\bc + t(1-C\lambda_n)\|\bc\|^2_2\Big] = \frac{(1-C\lambda_n)^2 \|\bc\|^4_2}{\bc^\top\hat{\boldsymbol \Gamma}\bc}.
\end{align*}
This provides a lower bound on $\hat{\boldsymbol u}^\top\hat{\boldsymbol \Gamma} \hat{\boldsymbol u}$. Next, we will use the bound to further provide a lower bound on $\tilde V$. Define
\begin{equation}\label{eq:define_gamma}
    \boldsymbol{\Gamma}= \frac{2}{|S_2|} \sum_{i \in S_2} \mathbb{E} \Big[ f'(\tilde v_i) \bX_i \bX^\top_i\Big] 
\end{equation}
Note that, conditional on $S_1$,
\begin{align*}
    \boldsymbol c^\top\hat{\boldsymbol \Gamma}\boldsymbol c - \boldsymbol c^\top\boldsymbol\Gamma\boldsymbol c = \frac{2}{|S_2|} \sum_{i \in S_2} \left\{\boldsymbol c^\top  f'(\tilde v_i) \boldsymbol X_i\boldsymbol X_i^\top\boldsymbol c -  \mathbb E\big[\boldsymbol c^\top f'(\tilde v_i) \boldsymbol X_i\boldsymbol X_i^\top \boldsymbol c\big]\right\},
\end{align*}
so that $\boldsymbol c^\top\hat{\boldsymbol \Gamma}\boldsymbol c - \boldsymbol c^\top\boldsymbol\Gamma\boldsymbol c \stackrel{P}{\to} 0$ by the weak law of large numbers. Moreover, since $\boldsymbol\Gamma$ is positive definite by~\eqref{eq:Gamma_pos_def} with minimum eigenvalue bounded away from $0$, 
we obtain $\frac{\boldsymbol c^\top\hat{\boldsymbol\Gamma}\boldsymbol c}{\boldsymbol c^\top\boldsymbol\Gamma\boldsymbol c} \stackrel{P}{\to} 1$. Since $\lambda_n \to 0$ by our sparsity assumption of the theorem, for $d, n$ large enough and small $\epsilon>0$,
\begin{align*}
    \hat{\boldsymbol u}^\top\hat{\boldsymbol \Gamma} \hat{\boldsymbol u} \geq \frac{(1-C\lambda_n)^2 \|\bc\|^4_2}{\boldsymbol c^\top\hat{\boldsymbol\Gamma}\boldsymbol c} > \frac{(1-\epsilon)(1-C\lambda_n)^2\|\bc\|^4_2}{\boldsymbol c^\top\boldsymbol\Gamma\boldsymbol c} > \frac{0.99 \|\bc\|^4_2 }{\boldsymbol c^\top\boldsymbol\Gamma\boldsymbol c} \geq \frac{0.99 \|\bc\|^2_2}{\lambda_{\text{max}}(\boldsymbol\Gamma)}
\end{align*}
\noindent with probability converging to 1. 
Since $\tilde V = |S_2|^{-1}\hat{\boldsymbol u^\top}\hat{\boldsymbol\Gamma}\hat{\boldsymbol u}$, it follows that for $ n$ large enough, 
\begin{equation}\label{eq:v_lower_bound}
    V > \frac{0.99  \|\bc\|^2_2}{\lambda_{\text{max}}(\boldsymbol\Gamma) |S_2|} \ge C_1  \|\bc\|^2_2/n
\end{equation}
with high probability. Here, we have used the fact $\|f'\|_{\infty}\le 2$, and $\lambda_{\max}(\bSigma)= O(1)$ to obtain $\lambda_{\text{max}}(\boldsymbol\Gamma)=O(1)$. This proves the required lower bound of $V$, completing the proof of $(i)$.

\noindent \textbf{Proof of $(ii)$:}

The conclusion will follow from the definition of $\hat \bu$ once we show such a projection vector $\hat\bu$ exists. It suffices to show that conditioned on $\bX_{-{S_2}}$, the matrix $ \boldsymbol \Gamma$ defined by~\eqref{eq:define_gamma}
is invertible and $\bu^\star= \boldsymbol \Gamma^{-1}\bc$ is feasible for the constraints
\begin{align}
        &\Big\|\Big(\frac{1}{|S_2|}\sum_{i \in S_2} f'(\tilde v_i)\bu^{\star^\top}\bX_i\bX_i^\top\Big) - \boldsymbol c^\top\Big\|_\infty = O(\|\bc\|_2\lambda_n), \quad 
        \max\limits_{i \in |S_2|} \big|\bX_i^\top \bu^\star\big| = O(\|\bc\|_2 \sqrt{\log n}), \nonumber\\
        &\Big|\Big(\frac{1}{|S_2|}\sum_{i \in S_2} f'(\tilde v_i)\bu^{\star^\top}\bX_i\bX_i^\top\bc\Big) - \|\bc\|^2_2\Big| = O(\|\bc\|^2_2\lambda_n). \label{eq:feasibility-2}
    \end{align}
with probability at least $1- C_1 n^{-2}$.

To show invertibility of the matrix $\boldsymbol \Gamma$, note that $f'>0$. Hence, we want to show that for all $\boldsymbol z \in \mathbb R^d$ such that $\|\boldsymbol z\|_2=1$, there exists some $\delta>0$ such that
\begin{align}\label{eq:Gamma_pos_def}
    \boldsymbol{z^\top\Gamma z} = \frac{2}{|S_2|} \sum_{i \in S_2} \mathbb{E} \Big[ f'(\tilde v_i) (\bz^\top\bX_i)^2\Big] >2 \delta.
\end{align}
Rewriting \eqref{eq:Gamma_pos_def} as the double of the average of
\begin{align*}
    \mathbb E\big[f'(\tilde v_i) (\boldsymbol{X_i^\top z})^2\big] = \mathbb E\big[f'(v_i)(\boldsymbol X_i^\top \boldsymbol z)^2\big] - \mathbb E\Big[\big(f'(v_i) - f'(\tilde v_i)\big)(\boldsymbol X_i^\top \boldsymbol z)^2\big],
\end{align*}
it is enough to show that
\begin{align}
    &\mathbb E\big[f'(v_i)(\boldsymbol X_i^\top \boldsymbol z)^2\big] > 2\delta, \quad \mathbb E\Big[\big(f'(v_i) - f'(\tilde v_i)\big)(\boldsymbol X_i^\top \boldsymbol z)^2\Big] < \delta. \label{eq:pos_def_steps}
\end{align}
for each $i \in S_2$. To this end, note that for any $T>0, \mathbb E[f'(v_i)(\boldsymbol X_i^\top \boldsymbol z)^2] \geq E\big[f'(v_i)(\boldsymbol X_i^\top \boldsymbol z)^2 \mathbbm{1}\{|v_i| < T\}\big]$. Notice that $f'$ is positive, symmetric around $0$ and strictly decreasing for $x>0$. Therefore, there exists $C>0$ such that $f'(x) \ge C$ for any $|x|<T$. Hence,

\begin{align*}
    \mathbb E[f'(v_i) (\boldsymbol{X_i^\top z})^2] &\geq C \mathbb E[(\boldsymbol{X_i^\top z})^2 \mathbbm{1}\{|v_i|<T\}] \\
    &= C\Big(\mathbb E[(\boldsymbol{X_i^\top z})^2] - \mathbb E[(\boldsymbol{X_i^\top z})^2 \mathbbm{1}\{|v_i| \geq T\}]\Big) \\
    &\geq C_1 -C_2 \Big[\mathbb P\big(|v_i| \geq T\big)\Big]^{1/2} \\
    &= C_1 -C_2 \Big[\mathbb P\big(|m_i(\by)+ \bX^\top_i \btheta| \geq T\big)\Big]^{1/2}\\
    & \geq C_1 - C_2\sqrt{2 \mathbb P\big(|\boldsymbol{X_i^\top\theta}| \geq T - |m_i(\by)|)} \\
    &\geq 2\delta,
\end{align*}
 for large $T>0$, since $\bX^\top_i \btheta$ is sub-Gaussian random variable and $|m_i(\by)|$ is bounded. This yields the first inequality in~\eqref{eq:pos_def_steps}. To establish the second inequality in \eqref{eq:pos_def_steps}, observe that $(f'(\cdot))^2$ is Lipschitz continuous yielding
\begin{align*}
    \Big|\mathbb E\big[f'(v_i)-f'(\tilde v_i)\big]\big(\boldsymbol{z^\top X_i})^2\big]\Big| &\lesssim \sqrt{\mathbb E\Big[\big(\boldsymbol X_i^\top(\boldsymbol\theta-\tilde{\boldsymbol\theta})\big)^2\Big]} \sqrt{\mathbb E\big[(\boldsymbol{z^\top X_i})^4\big]} \\
    & \lesssim \sqrt{\lambda_{\text{max}}(\boldsymbol{\Sigma}) \|\tilde{\boldsymbol\theta}-\boldsymbol\theta\|_2^2} \\
    &\lesssim \sqrt{\frac{s \log d}{n}}.
\end{align*}
Here, the first inequality is Cauchy-Schwarz, second inequality follows from Assumption \ref{assn:covariates}, and the final inequality from Proposition~\ref{thm:first_step}). This establishes the second inequality in \eqref{eq:pos_def_steps} and thus shows $\boldsymbol{\Gamma}$ is invertible. 

Turning to the proof of~\eqref{eq:feasibility-2}, set $\bu^* \in \mathbb{R}^d$ such that $\bu^\star = \boldsymbol\Gamma^{-1} \bc$. Substituting the value of $\boldsymbol\Gamma$ provides that
\begin{align*}
    (\bu^*)^\top \frac{2}{|S_2|} \sum_{i \in S_2} \mathbb E[f'(\tilde v_i) \boldsymbol{X_iX_i^\top}] = \boldsymbol c^\top.
\end{align*}
Letting $\alpha_{ik}:= f'(\tilde v_i) (\bu^*)^\top \boldsymbol{X_iX_i}^\top\boldsymbol e_k$, it follows that $\mathbb E[\alpha_{ik}] = c_k$. Since the $\{\boldsymbol X_i\}_{i \in [n]}$ are sub-Gaussian, and $(f')^2$ is uniformly bounded above by a constant, it follows that $\alpha_{ik}$ is a sub-exponential random variable for all $(i,k) \in [n] \times [d]$, and this implies that $(\bu^*)^\top f'(\tilde v_i) \boldsymbol{X_iX_i^\top} - \boldsymbol c^\top$ is a centered sub-exponential random vector. By applying a Bernstein-type concentration inequality for sub-exponential variables (e.g. \cite[Proposition 5.16]{vershynin2010introduction}), it follows that for all $C > 0$,
\begin{align*}
    \mathbb P\Big(\Big|\frac{2}{|S_2|}\sum_{i \in S_2}\big((\bu^*)^\top f'(\tilde v_i) \boldsymbol{X_iX_i^\top} - \boldsymbol c^\top\big)_k\Big|\geq C\sqrt{\frac{\log d}{n}}\Big) \leq 2 \exp\Big[-C_2 \log d \Big],
\end{align*}
By a union bound, we obtain that with probability $1-d^{-c}$ for some $c>0$, we have 
\begin{align*}
    \Big\|\frac{2}{|S_2|}\sum_{i \in S_2}\big((\bu^*)^\top f'(\tilde v_i) \boldsymbol{X_iX_i^\top} - \boldsymbol c^\top\big)\Big\|_\infty \le  \lambda_n
\end{align*}
A similar concentration bound provides the second constraint. Moreover, by~\eqref{eq:Gamma_pos_def}, we know that $\lambda_{\max}(\boldsymbol{\Gamma}^{-1})=O(1)$, implying $\|\bu^\star\|_2=O(\|\bc\|_2)$. Hence $\bX^\top_i \bu^\star$ are i.i.d. sub-Gaussian random variables with sub-Gaussian norm $O(\|\bc\|_2)$. Hence, with probability $1- n^{-c}$ we, 
\[
\max_{i \in S_2} |\bX^\top_i \bu^\star| \lesssim \|\bc\|_2\sqrt{\log n}.
\]
Hence $u^*$ satisfies the constraints in (ii),  completing the proof of~\eqref{eq:feasibility-2}.

\noindent \textbf{Proof of $(iii)$:}

We begin by noting that $|\mathcal{R}_i| \le 8 (\bX^\top_i(\btheta- \tilde\btheta))^2$, where $\mathcal{R}_i$ is defined by~\eqref{eq:TS1}. Hence we have,
\begin{align*}
    \left|\frac{1}{|S_2|}\sum_{i \in S_2} \mathcal{R}_i \hat\bu^\top \bX_i \right| &\lesssim \max_{i \in S_2}|\hat\bu^\top \bX_i| \frac{1}{|S_2|}\sum_{i \in S_2}  (\btheta- \tilde\btheta) \bX_i\bX^\top_i(\btheta- \tilde\btheta) \\
    &\lesssim \|\bc\|_2 \sqrt{\log n} \|\btheta- \tilde\btheta\|^2_2 \\
    &= O\left(\|\bc\|_2\frac{s \log d \sqrt{\log n}}{n}\right),
\end{align*}
where the second inequality follows from~\eqref{eq:constrained set} and the fact that $\bSigma$ has bounded eigenvalues (Assumption~\ref{assn:covariates}), and final equality follows from the consistency of $\tilde \btheta$ (Proposition~\ref{thm:first_step}).

\subsection*{Proof of Proposition~\ref{thm:first_step}}

We begin by stating the following definition of Rademacher complexity common in high-dimensional statistics:  Given the covariates $\bX_1, \ldots, \bX_n$, define the Rachemacher complexity $\mathscr{R}_n$ as 
\begin{equation}\label{eq:defn_rademacher}
    \mathscr{R}_n:= \mathbb{E}\left(\left\|\frac{1}{n} \sum_{i=1}^n \varepsilon_i \bX_i \right\|_{\infty}\right),
\end{equation}
where $\varepsilon_1,\ldots, \varepsilon_n$ are i.i.d. Rademacher random variables. Note that the expectation above is taken with respect to the randomness of $\varepsilon_i$'s and $\bX_i$'s together. 

\noindent Recall the definition of pseudo-likelihood on an independent set $S_1$ given by~\eqref{eq:define_LS1}. The proof of Proposition~\ref{thm:first_step} consists of two key steps: 

i. provide high probability upper bound for the $\ell_\infty$ norm of the gradient $\|\nabla L_{S_1}(\btheta)\|_{\infty}$ (Lemma~\ref{lemma:grad_conc}), 

ii. prove that log-likelihood $L_{S_1}$ is restricted strong convex with high probability (Lemma~\ref{lemma:concavity}).

\noindent Given these two results, the conclusion of the proposition follows by invoking \cite[Corollary 9.20]{Wainwright_2019} (Corollary 9.20) in the following manner: Set $\lambda_n= C \sqrt{\frac{\log d}{n}}$. Since $\bX_i$'s are i.i.d. sub-Gaussian random variables, we know~\cite[Exercise 9.8]{Wainwright_2019} that the Rademacher complexity satisfies $\mathscr{R}_n= O(\sqrt{\frac{\log d}{n}})$. Therefore, under the assumptions of the Proposition~\ref{thm:first_step}, we have $s \mathscr{R}_n= o(1)$. Hence, if $\|\nabla L_{S_1}(\btheta)\|_{\infty} \le \frac{\lambda_n}{2}$, we have 
\begin{align*}
    \|\tilde\btheta-\btheta\|_2^2 \lesssim s\lambda^2_n \lesssim s\frac{\log d}{n} \; \; \text{and} \; \; \|\tilde\btheta-\btheta\|_1 \lesssim s\lambda_n \lesssim s \sqrt{\frac{\log d}{n}}.
\end{align*}
Since $\|\nabla L_{S_1}(\btheta_n)\|_\infty \leq \frac{\lambda_n}{2}$ occurs with probability $1-o(1)$ by Lemma~\ref{lemma:grad_conc}, these bounds in the above display hold with probability $1-o(1)$ as well, completing the proof of Proposition \ref{thm:first_step}.

Next, we state the two main Lemmas, and their proofs afterwards.

\begin{lemma}(Concentration of the gradient)\label{lemma:grad_conc}: 
Suppose the conditions of Proposition~\ref{thm:first_step} hold. There exists $C>0$ such that with $\lambda_n=C \sqrt{\log d/n}$,
\begin{align*}
\mathbb P\bigg(\|\nabla L_{S_1}(\boldsymbol\theta)\|_\infty>\frac{\lambda_n}{2}\bigg) = o(1),
\end{align*}
where the $o(1)$ term goes to $0$ as $n,d \to \infty$.
\end{lemma}

\begin{lemma}(Restricted strong convexity of the pseudo-likelihood)\label{lemma:concavity}:
Suppose the assumptions of Theorem \ref{thm:first_step} hold. Then, there exist positive constants $\nu, c$ such that
\begin{align*}
L_{S_1}(\tilde\btheta) - L_{S_1}(\boldsymbol\theta) - \nabla L_{S_1}(\boldsymbol\theta)^\top (\tilde\btheta-\btheta) \geq \nu\|\boldsymbol\eta\|_2^2 - c\mathcal R_n \|\boldsymbol\eta\|_1^2,
\end{align*}
with probability at least $1-o(1)$, for all $\boldsymbol\eta \in \mathbb R^d$ such that $\|\boldsymbol\eta\|_2 \le 1$. Here, $\mathscr{R}_n$ is defined by~\eqref{eq:defn_rademacher}
\end{lemma}

Now we provide the proofs of the above two Lemmas.

\noindent \textbf{Proof of Lemma~\ref{lemma:grad_conc}:}

\noindent Consider the good set 
\begin{equation}\label{eq:defn_good_set}
    \mathcal{A}:= \left\{\max_{k \in [d]} \frac{1}{n} \sum_{i=1}^n X^2_{ik} \le 2 \lambda_{\max}(\Sigma) \right\}
\end{equation}
We will prove the Lemma under the assumption $\mathcal{A}$ holds. In the end of the proof, we will show that $\mathbb{P}(\mathcal{A})= 1- o(1)$. Throughout the proof we will use $c,C,C_i$'s to denote arbitrary positive constants.

We recall from~\eqref{eq:size_set} that $|S_1| \ge cn$ for some $c>0$ fixed. From the definition of $L_{S_1}$ in~\eqref{eq:define_LS1}, we obtain that
\begin{equation}\label{eq:ls1_gradient}
    (\nabla L_{S_1}(\btheta))_k= -\frac{1}{|S_1|} \sum_{i \in S_1} X_{ik}   \left[y_i - \tanh(v_i)\right],
\end{equation}
where $v_i$ is defined as in~\eqref{eq:define_vi}. Therefore, we need to show that
\begin{equation*}
    \mathcal{T}:= \max_{k \in [d]} \left| \frac{1}{|S_1|} \sum_{i\in S_1} X_{ik}   \left[y_i - \tanh(v_i)\right] \right| \le C \sqrt{\frac{\log d}{n}}
\end{equation*}
for some $C>0$ with high probability. To this end, note that each individual component $X_{ik}   \left[y_i - \tanh(v_i)\right]$ satisfies
\begin{equation*}
    \mathbb{E} \left[X_{ik}   \left[y_i - \tanh(v_i)\right]\right]=0.
\end{equation*}
Further, conditioned on $\by_{-S}$, $y_i$'s are independent for $i \in S_1$. Hence for each $k$, $(\nabla L_{S_1}(\btheta))_k$ consists of a weighted sum of independent sub-Gaussian random variables. Therefore
\begin{align*}
    \mathbb{P}\left( \mathcal{T} > C \sqrt{\frac{\log d}{n}} \right) & \le d \max_{k \in [d]} \mathbb{P} \left( \left| \frac{1}{|S_1|} \sum_{i\in S_1} X_{ik}   \left[y_i - \tanh(v_i)\right] \right| > C \sqrt{\frac{\log d}{n}} \right) \\
    & \le d \max_{k \in [d]} \exp\Bigg(- C_2 \frac{n \log d }{\sum_{i \in S_1} X^2_{ik}}\Bigg) \\
    & \le d \exp\Bigg(- C_3 \frac{\log d }{\max_{k \in [d]}\frac{1}{n}\sum_{i=1}^n X^2_{ik}}\Bigg) \rightarrow 0,
\end{align*}
for large $C$ under the event $\mathcal{A}$. This completes the proof once we show $\mathbb{P}(\mathcal{A})= 1-o(1)$.

The bound on $\mathbb{P}(\mathcal{A})$ follows by sub-exponential concentration bounds followed by a union bound. To see this, note that for any $k \in [d]$
\begin{align*}
    \frac{1}{n}\sum_{i=1}^{n} X^2_{ik}= \frac{1}{n}\sum_{i=1}^{n} \Big(X^2_{ik} -\mathbb{E}(X^2_{1k})\Big)+ \Sigma_{kk}
\end{align*}
Since $\bX_i$ are sub-Gaussian random vectors, $X^2_{ik}$, $i \in [n]$ are i.i.d. sub-exponential random variables with sub-exponential norm bounded by some $C>0$. Therefore, by standard sub-exponential concentration inequality~\cite[Corollary 5.17]{vershynin2010introduction}, we have 
\[\mathbb{P}\left(\left|\frac{1}{n}\sum_{i=1}^{n} \Big(X^2_{ik} -\mathbb{E}(X^2_{1k})\Big)\right| \ge \varepsilon \right)\le \exp \Bigg( -C \min\{\varepsilon,\varepsilon^2\}n \Bigg)\]
Choosing $\varepsilon= C \sqrt{\frac{\log d}{n}}$ and taking union bound over $k \in [d]$ we obtain that with probability $1-o(1)$, we have 
\begin{align*}
    \frac{1}{n}\sum_{i=1}^{n} X^2_{ik} \le C\sqrt{\frac{\log d}{n}}+ \max_{k \in [d]} \Sigma_{kk} \le 2 \lambda_{\max}(\Sigma),
\end{align*}
where the final inequality uses Assumption~\ref{assn:covariates} and by noting that for any $k$, $\bSigma_{kk} \le \lambda_{\max}(\bSigma)$. This completes the proof of Lemma~\ref{lemma:grad_conc}.

\noindent \textbf{Proof of Lemma~\ref{lemma:concavity}:}
Similar to the proof of Lemma~\ref{lemma:grad_conc}, we will assume that $|S_1| \ge c n$ for some $c>0$. Using the definition of $L_{S_1}$ given by~\eqref{eq:define_LS1}, we obtain that
\begin{equation}\label{eq:LSI_hessian}
    \nabla^2 L_{S_1}(\btheta)= \frac{1}{|S_1|} \sum_{i \in S_1} \bX_i \bX^\top_i  \sech^2 (v_i),
\end{equation}
where $v_i$ is defined as~\eqref{eq:define_vi}. Using a Taylor's expansion of $L_{S_1}$ around $\btheta$, we obtain that
\begin{align*}
L_{S_1}(\btheta+ \boldsymbol{\eta}) = L_{S_1}(\boldsymbol\theta) + \nabla L_{S_1}(\boldsymbol\theta)^\top \boldsymbol\eta + \frac{1}{2}\boldsymbol\eta^\top \nabla^2L_{S_1}(\boldsymbol\theta + t\boldsymbol\eta)\boldsymbol\eta,
\end{align*}
for some $t \in (0,1)$. Using the expression for $\nabla^2 L_{S_1}(\boldsymbol\theta)$ given by~\eqref{eq:LSI_hessian}, it follows that
\begin{align*}
     &L_{S_1}(\boldsymbol{\theta+\eta}) - L_{S_1}(\boldsymbol\theta) - \nabla L_{S_1}(\boldsymbol\theta)^\top\boldsymbol\eta \\
&= \underbrace{\frac{1}{2|S_1|}\sum_{i \in S_1} (\boldsymbol{\eta}^\top \bX_i)^2 \sech^2(  m_i(\by) +  \bX^\top_i(\btheta+ t\boldsymbol{\eta})))}_{=:\mathcal{T}_1} 
\end{align*}
The remainder of the proof is to show 
$\mathcal{T}_1$ satisfies the required lower bound given by Lemma~\ref{lemma:concavity}. 
To this end, note that
\begin{align}\label{eq:t2_i}
    \mathcal{T}_1 &= \frac{1}{2|S_1|}\sum_{i \in S_1} (\boldsymbol{\eta}^\top \bX_i)^2 \sech^2\big(  m_i(\by) +   \bX_i^\top(\boldsymbol\theta + t\boldsymbol\eta)\big) \nonumber\\
    &\geq \frac{1}{2|S_1|}\sum_{i \in S_1} (\boldsymbol{\eta}^\top \bX_i)^2  \sech^2\big(C_1+ C |\bX_i^\top(\boldsymbol\theta + t\boldsymbol\eta)|\big),
\end{align}
for some constants $C_1, C>0$. Here the inequality holds since $\sup_{i} m_i(\by)= O(1)$ by Assumption~\ref{assn:graph}. 
Therefore, we need to lower bound the final quantity of~\eqref{eq:t2_i}. Recall the definition of Rademacher complexity from~\eqref{eq:defn_rademacher}. By invoking~\cite[Proposition 19]{mukherjee2024logistic}, we immediately obtain that there exists $\nu, c_0>0$ such that 
\begin{align*}
   \mathcal{T}_1 \geq \nu\|\boldsymbol\eta\|_2^2 - c\mathscr{R}_n \|\boldsymbol\eta\|_1^2,
\end{align*}
with probability at least $1-o(1)$, for all $\boldsymbol\eta \in \mathbb R^d$ such that $\|\boldsymbol\eta\|_2 \le 1$. We note that the lower bound provided by~\cite[Proposition 19]{mukherjee2024logistic} is based on an average over all $i \in [n]$. However, the same proof works verbatim once we note that $|S_1| \ge cn$ by the bounded degree condition (Assumption~\ref{assn:graph}). This completes the proof the Lemma.

\subsection*{Proof of Proposition~\ref{prop:CI}}

To prove validity of our confidence interval, it is enough to show that \begin{align}\label{eq:hatv_consistency}
    \widehat V/V \xrightarrow{ \mathbb{P}}1,
\end{align} where $\widehat V$ and $V$ are defined by~\eqref{eq:estimated_var} and~\eqref{eq:define_vj} respectively. To this end, note that 
\begin{align*}
    \left|\frac{\widehat{V}}{V}-1\right|&= \frac{1}{V} \left|\frac{4}{|S_2|^2} \sum_{i \in S_2} (f(v_i)(1-f(v_i))- f(\tilde v_i)(1-f(\tilde v_i)))(\hat\bu^\top \bX_i)^2 \right|\\
    & \le \max_{i \in S_2}\left| \frac{f(v_i)(1-f(v_i))}{f(\tilde v_i)(1-f(\tilde v_i))} -1\right|
\end{align*}
We have to show the above random variable converges to $0$ in probability. Since $f(x)+ f(-x)=1$ by definition, it is enough to show
\begin{align}\label{eq:two_ineq}
    \max_{i \in S_2}\left| \frac{f(|v_i|)}{f(|\tilde v_i|)} -1\right| \xrightarrow{\mathbb{P}} 0, \qquad \max_{i \in S_2}\left| \frac{1-f(|v_i|)}{1-f(|\tilde v_i|)} -1\right| \xrightarrow{\mathbb{P}} 0
\end{align}
To show the first inequality above, recall that $f(x)= e^{x}/(e^x+ e^{-x})$ by~\eqref{eq:define_f}. Since $f$ is increasing in $\mathbb R$, we have $f(x) \ge 1/2$ for all $x \ge 0$. Further $f'(x)= 2e^{2x}/(e^{2x}+1)^2$, implying that $\|f'\|_{\infty} \le 2$. Therefore we have
\begin{align*}
     \max_{i \in S_2}\left| \frac{f(|v_i|)}{f(|\tilde v_i|)} -1\right| &= \max_{i \in S_2}\left| \frac{f(|v_i|)- f(|\tilde v_i|)}{f(|\tilde v_i|)}\right| \\
     &\le 2\|f'\|_{\infty}\max_{i \in S_2} |v_i - \tilde v_i|\\
     &\le 4 \max_{i \in S_2} |\bX^\top_i(\btheta- \tilde\btheta)| \\
     &\lesssim \sqrt{\log n} \|\btheta- \tilde\btheta\|_2 \\
     &\le \sqrt{\frac{s \log d \log n}{n}} = o(1),
\end{align*}
with high probability. Here the third inequality follows since the maximum of $n$ independent sub-Gaussian random variables is upper bounded by $C\sqrt{\log n}$ with high probability and the fourth inequality follows from Proposition~\ref{thm:first_step}.

To show the second inequality of~\eqref{eq:two_ineq}, note that $f(-|\tilde v|)= 1-f(|\tilde v_i|)$. Therefore we have
\begin{align*}
    \max_{i \in S_2}\left| \frac{1-f(|v_i|)}{1-f(|\tilde v_i|)} -1\right| &= \max_{i \in S_2}\left| \frac{f(|\tilde v_i|)-f(|v_i|)}{f(-|\tilde v_i|)} \right|\\
    &\lesssim \sqrt{\frac{s \log d \log n}{n}} \max_{i \in S_2} \frac{1}{f(-|\tilde v_i|)},
\end{align*}
where the inequality follows from the above argument. Hence, it is enough to bound $\max_{i \in S_2} \frac{1}{f(-|\tilde v_i|)}$. To this end note that,
\begin{align*}
    \max_{i \in S_2} \frac{1}{f(-|\tilde v_i|)}= \max_{i \in S_2} \frac{e^{-|\tilde v_i|}+e^{|\tilde v_i|}}{e^{-|\tilde v_i|}} \lesssim e^{C_1 +C_2 \max_{i \in S_2} |\bX^\top_i \tilde\btheta|},
\end{align*}
for some $C_1,C_2 >0$. Further, we have 
\begin{align*}
    \max_{i \in S_2} |\bX^\top_i \tilde\btheta| =  \max_{i \in S_2} |\bX_i^\top(\tilde\btheta-\btheta) + \bX_i^\top\btheta|\le o_{\mathbb{P}}(\sqrt{\log n})+ \max_{i \in S_2} |\bX^\top_i \btheta|= O_{\mathbb{P}}(\sqrt{\log n}),
\end{align*}
since $\|\btheta\|_2=O(1)$ by Assumption~\ref{assn:theta}. Therefore, for some $C>0$,
\begin{align*}
    \max_{i \in S_2} \frac{1}{f(-|\tilde v_i|)} \lesssim e^{ C \sqrt{\log n}}
\end{align*}
with high probability. Hence,
\begin{align*}
    \max_{i \in S_2}\left| \frac{1-f(|v_i|)}{1-f(|\tilde v_i|)} -1\right| \lesssim \sqrt{\frac{s \log d \log n}{n}} e^{ C \sqrt{\log n}} =o(1)
\end{align*}
with high probability using our sparsity assumption in Theorem~\ref{thm:main}. This shows that $\widehat{V}/V \xrightarrow{\mathbb{P}}1$. Hence, by Slutsky's theorem, we have the validity of the confidence interval. 

Next we prove the required upper bound on the length of the confidence interval. Recall the definition of $\mathcal{C}$ given by~\eqref{eq:define_conf_int}. The length of the confidence interval is exactly $2 z_{\alpha/2}\sqrt{\widehat{V}}$. We have shown above that $\widehat{V}/V \xrightarrow{\mathbb{P}}1$. Hence the length of the confidence interval is $(1+o_{P}(1))2 z_{\alpha/2}\sqrt{V}$. This completes the proof of the Proposition.

\end{document}